\documentclass[a4paper,12pt]{amsart}
\usepackage{amsmath,amssymb,amsthm,enumerate,
textcomp,mathabx}
\usepackage[T2A]{fontenc}
\usepackage[utf8]{inputenc}
\usepackage{cite}

\usepackage{graphicx}

\DeclareMathOperator{\Aut}{Aut}

\DeclareMathOperator{\Ann}{Ann}

\newtheorem{theorem}{Theorem}
\newtheorem{lemma}{Lemma}
\newtheorem{proposition}{Proposition}
\newtheorem{definition}{Definition}
\newtheorem{corollary}{Corollary}
\newtheorem{note}{Remark}
\newtheorem{example}{Example}

\begin{document}
\sloppy
\title[On Lie algebras associated with modules over polynomial rings]
{On Lie algebras associated with modules\\ over polynomial rings}
\author
{A.P. Petravchuk,  K.Ya. Sysak}
\address{A.P. Petravchuk:
Department of Algebra and Mathematical Logic, Faculty of Mechanics and Mathematics,
Taras Shevchenko  National University of Kyiv, 64, Volodymyrska street, 01033  Kyiv, Ukraine}
\email{aptr@univ.kiev.ua , apetrav@gmail.com}
\address{K.Ya. Sysak:
Department of Algebra and Mathematical Logic, Faculty of Mechanics and Mathematics,
Taras Shevchenko National University of Kyiv, 64, Volodymyrska street, 01033  Kyiv, Ukraine}
\email{sysakkya@gmail.com}
\date{\today}
\keywords{Lie algebra, polynomial ring,  metabelian algebra, module, weak isomorphism, linear operator}
\subjclass[2000]{Primary 17B30, 17B60; Secondary 13C13}

\begin{abstract}
Let $\mathbb K$ be an algebraically closed field of characteristic zero. Let $V$ be a module over the polynomial ring  $\mathbb K[x,y]$. The actions of $x$ and $y$ determine linear operators $P$ and $Q$ on $V$  as a vector space over $\mathbb K$. Define the  Lie algebra $L_V=\mathbb K\langle P,Q\rangle \rightthreetimes V$ as the semidirect product of two abelian Lie algebras with the natural action of $\mathbb K\langle P,Q\rangle$ on $V$. We show that if $\mathbb K[x,y]$-modules $V$ and $W$ are isomorphic or weakly isomorphic, then the corresponding associated Lie algebras $L_V$ and $L_W$ are isomorphic. The converse is not true: we construct two $\mathbb K[x,y]$-modules $V$ and $W$  of dimension $4$ that are not weakly isomorphic but their   associated Lie algebras  are isomorphic. We characterize such pairs of $\mathbb K[x, y]$-modules of arbitrary dimension.  We prove that indecomposable modules  $V$ and $W$  with $\dim V=\dim W\geq 7$ are weakly isomorphic if and only if their associated Lie algebras  $L_V$ and $L_W$ are isomorphic.
\end{abstract}

\maketitle

\section{Introduction}

Let $\mathbb K$ be an algebraically closed field of characteristic zero. Let $V$ be a module over the polynomial ring $\mathbb K[x,y]$. Define the commuting  linear operators $P$ and $Q$ on $V$ by setting $P(v)=x\cdot v$ and $Q(v)=y\cdot v$ for all $v\in V$. Conversely, if $P$ and $Q$ are commuting linear operators on  $V$, then the  vector space $V$ can be considered as the $\mathbb K[x, y]$-module with multiplication $f(x,y)\cdot v=f(P,Q)(v)$ for all $v\in V$ and $f(x,y)\in \mathbb K[x,y]$.

For each $\mathbb K[x,y]$-module $V$, let us construct the metabelian Lie algebra $L_V=\mathbb K\langle P,Q\rangle \rightthreetimes V$, which is the external semidirect product of the abelian Lie algebra $\mathbb K\langle P,Q\rangle$ of dimension $2$ and the abelian Lie algebra $V$  with the natural action of $\mathbb K\langle P,Q\rangle$  on $V$. We say that the Lie algebra $L_V$ \emph{is  associated with} the $\mathbb K[x, y]$-module $V$.

Modules over polynomial rings were studied by many authors: Gelfand and Ponomarev \cite{GP} proved that the problem of classifying finite dimensional modules over $\mathbb K[x, y]$ contains the problem of classifying matrix pairs up to similarity. Quillen \cite{Qui} and Suslin \cite{Sus} studied projective modules over polynomial rings in connection with Serre's problem.

Our goal is to study relations between   finite dimensional  $\mathbb K[x,y]$-modules $V$  and  the corresponding associated Lie algebras $L_V.$

For each automorphism  $\theta $ of the abelian Lie algebra $\mathbb K\langle P,Q\rangle $ with
$$\theta (P)=\alpha _{11}P+\alpha _{12}Q,  \qquad  \theta (Q)=\alpha _{21}P+\alpha _{22}Q, \alpha _{ij}\in \mathbb K ,$$
the semidirect product $\mathbb K\langle \theta (P), \theta (Q)\rangle \rightthreetimes V$ is isomorphic to $L_V.$ The corresponding transformation of the ring $\mathbb K[x, y]$ is the automorphism of $\mathbb K[x, y]$ defined by  $\theta (x)=\alpha _{11}x+\alpha _{12}y $ and $ \theta (y)=\alpha _{21}x+\alpha _{22}y. $ This automorphism defines the ``twisted''  module $V_{\theta}$ with the following multiplication on the vector space  $V:$
$$x\circ v=\theta(x)\cdot v \quad\text{and}\quad y\circ v=\theta(y)\cdot v\qquad \mbox{for all} \ v\in V.$$  The modules $V$ and $V_{\theta}$ are called \emph{weakly isomorphic};  this notion in matrix form was studied by Belitskii, Lipyanski, and Sergeichuk \cite{BLS}. We show in Proposition \ref{propos} that
if  $\mathbb K[x, y]$-modules $W$ and $ V $ are isomorphic  (or even weakly isomorphic), then the associated Lie algebras $L_V$ and $L_W$ are isomorphic. The converse statement is  not true: Lemma \ref{example} gives an example of $\mathbb K[x, y]$-modules  $V$ and $W$ that are not weakly isomorphic, but their Lie algebras $L_V$ and $L_W$ are isomorphic. Nevertheless, if $V$ and $W$ are indecomposable modules with $\dim _{\mathbb K}V=\dim _{\mathbb K}W\geq 7,$ then by Theorem \ref{indecomposable} $V$ and $W$ are weakly isomorphic if and only if the associated Lie algebras  $L_V$ and $L_W$ are isomorphic.

Theorem \ref{structure} gives a characterization of  pairs  $(V,W)$ of $\mathbb K[x, y]$-modules that are not weakly isomorphic but their  associated Lie algebras  are isomorphic. It shows that the problem of classifying finite dimensional  Lie algebras of the form $L=B\rightthreetimes A$  with an abelian ideal $A$ and a two-dimensional abelian subalgebra $B$ is equivalent to the problem of classifying finite dimensional $\mathbb K[x, y]$-modules  up to weak isomorphism. We think that the latter problem is wild. The wildness of some classes of metabelian Lie algebras was established by Belitskii, Bondarenko, Lipyanski, Plachotnik, and Sergeichuk  \cite{BBLPS,BLS}.

From now on, $\mathbb K$  is an algebraically closed field of characteristic zero. All Lie algebras and modules over $\mathbb K[x, y]$ that we consider are finite dimensional over $\mathbb K$.

Let  $L_1$ and $L_2$ be Lie algebras over a field $\mathbb K $. Let $\varphi : L_1 \to {\rm Der}{L_2}$ be a homomorphism of Lie algebras, in which ${\rm Der}{L_2}$ is the Lie algebra of all $\mathbb K$-derivations of $L_2.$ The \emph{external semidirect product} of $L_1$ and $L_2$ (denoted by $L_1\rightthreetimes _{\varphi}L_2$ or $L_1\rightthreetimes L_2$) is the vector space $L_1\oplus L_2$ with the  Lie bracket $$[(a_1, b_1) , (a_2, b_2)]=([a_1, a_2], D_1(b_2)-D_2(b_1)),$$ where $D_1:=\varphi (a_1)$ and $D_2:=\varphi (a_2)$.

If a Lie algebra $L$ contains an ideal $N$ and a subalgebra $B$ such that $L=N+B$ and $ N\cap B=0,$  then $L$ is the \emph{internal semidirect product}  of  Lie algebras $B$ and $N$ (we write $L=B\rightthreetimes N$) with the natural homomorphism of the Lie algebra $B$ into the Lie algebra ${\rm Der}_{\mathbb K}N.$

Let $W$ be a subspace  of a vector space $V$ over $K.$ Elements  $v_1, \ldots , v_n\in V$ are called \emph{linearly independent over $W$} if  $v_1+W, \ldots , v_n+W$ are linearly independent in the quotient space $V/W.$ We write $V\simeq W$ if $V$ and $W$ are isomorphic  $\mathbb K[x, y]$-modules.

\section{ $\mathbb K[x, y]$-modules with isomorphic associated Lie algebras}

Belitskii, Lipyanski, and Sergeichuk \cite{BLS} considered a notion that is  weaker than similarity of  matrix pairs.  Their notion in the case of commuting matrices can be formulated in terms of $\mathbb K[x, y]$-modules as follows.
Let $\theta$ be a linear  automorphism of the polynomial ring $\mathbb K[x,y]$, defined by linear homogeneous polynomials. Let $V_{\theta}$ be the $\mathbb K[x, y]$-module defined above.

\begin{definition}  $\mathbb K[x,y]$-modules $V$ and $W$ are called \emph{weakly isomorphic} if there exists a series of $\mathbb K[x,y]$-modules $V_1:=V, V_2,\dots, V_k:=W$ such that for each $i=1,2,\dots,k$ either $V_i\simeq V_{i+1}$, or $V_{i+1}=(V_i)_{\theta_i}$ for some linear automorphism ${\theta_i}\in \Aut(\mathbb K[x,y])$.
\end{definition}

Each weak isomorphism  is an equivalence relation on the class  of all $\mathbb K[x,y]$-modules. Clearly, isomorphic $\mathbb K[x,y]$-modules are weakly isomorphic. The converse is false by the following example.

\begin{example}\label{nonisom}
Let $V$ be an $n$-dimensional  vector space over a field $\mathbb K$. Choose a basis of $V$ over $\mathbb K$ and take any nonzero $n\times n$ matrix $A$ with entries in $\mathbb K.$  Denote by $V$ the $\mathbb K[x, y]$-module with the underlying vector space $V$ and the action of $x$ and $y$ on $V$ determined by the matrix pair  $(A,0_n).$ Let $\theta \in \Aut(\mathbb K[x,y])$ be the automorphism  such that $\theta(x)=y$ and $\theta(y)=x$. Define the action of $\mathbb K[x, y]$ on $V_{\theta}$ by the matrix  pair $(0_n, A).$   Then the modules  $V$ and $V_{\theta}$ are not isomorphic since the pairs $(A,0_n)$ and $(0_n,A)$ are not similar.
\end{example}

\begin{note}
$\mathbb K[x,y]$-modules $V$ and $W$ are weakly isomorphic if and only if there exists a $\mathbb K[x,y]$-module $U$ such that $U\simeq V$ and $W=U_{\theta}$ for some linear automorphism $\theta$ of the ring $\mathbb K[x,y]$.
\end{note}

\begin{proposition}\label{propos}
If $V$ and $W$ are weakly isomorphic modules over the ring $\mathbb K[x,y]$, then the corresponding associated Lie algebras $L_V$ and $L_W$ are isomorphic.
\end{proposition}

\begin{proof}
An isomorphism $\varphi$ of $\mathbb K[x,y]$-modules $V$ and $W$ can be extend  to the isomorphism $\overline{\varphi}$ of the Lie algebras $L_V=\mathbb K\langle P,Q\rangle \rightthreetimes V$ and $L_W=\mathbb K\langle S,T\rangle \rightthreetimes W$ via $\overline{\varphi}(P)=S$ and $  \overline{\varphi}(Q)=T $, and further by linearity. Assume that $V$ and $W$ are weakly isomorphic but not isomorphic. Then there exists a $\mathbb K[x,y]$-module $U$ such that $U\simeq V$ and $U_{\theta}=W$ for some linear automorphism $\theta\in \Aut(\mathbb K[x,y])$ given by a nonsingular $2\times 2$ matrix $[a_{ij}]$ over  $\mathbb K$:
$$\theta  (x)=\alpha _{11}x+\alpha _{12}y,\qquad \theta (y)=\alpha _{21}x+\alpha _{22}y. $$
  It suffices to show that $L_U$ and $L_{U_{\theta}}$ are isomorphic Lie algebras.
Write $L_U=\mathbb K\langle T_1,T_2\rangle \rightthreetimes U$, where $T_1, T_2:U\to U$ are commuting linear operators that determine the actions of $x$ and $y$ on $U$. Then $L_{U_{\theta}}= \mathbb K\langle \theta(T_1),\theta(T_2)\rangle \rightthreetimes U$, where $\theta(T_1)=\alpha_{11}T_1+\alpha_{12}T_2$ and $\theta(T_2)=\alpha_{21}T_1+\alpha_{22}T_2$. It is easy to see that $\mathbb K\langle \theta(T_1),\theta(T_2)\rangle=\mathbb K\langle T_1, T_2 \rangle$. Thus,  $L_U$ and  $L_{U_{\theta}}$ is the same Lie algebra.
\end{proof}

\begin{note}\label{remark}\label{restriction}
Let $L_V=\mathbb K\langle P,Q \rangle \rightthreetimes V$ and $L_W=\mathbb K\langle S,T \rangle \rightthreetimes W$ be the Lie algebras associated with $\mathbb K[x,y]$-modules $V$ and $W$. If $\varphi$ is an isomorphism of Lie algebras $L_V$ and $L_W$ such that $\varphi(V)=W$, then the restriction of  $\varphi$  on $V$ is a weak isomorphism of modules $V$ and $W$.
\end{note}

Let us prove Remark \ref{remark}.
Let  $\varphi(P)=\alpha_1S+\alpha_2T+w_1$ and $ \varphi(Q)=\beta_1S+\beta_2T+w_2$, where $\alpha_i, \beta_j  \in \mathbb K $ and $ w_1, w_2\in W.$ Consider the weakly isomorphic module $W_{\theta}$ with the automorphism $\theta$ defined by $\theta (x)=\alpha _{1}x+\alpha _{2}y$ and  $\theta (y)=\beta _{1}x+\beta _{2}y$. We can assume that $\varphi(P)=S+w_1$ and $\varphi(Q)=T+w_2.$ Then the restriction $\overline{\varphi}$ of $\varphi$ on $V$ is an isomorphism of $\mathbb K[x, y]$-modules $V$ and $W_{\theta}.$

\begin{lemma}\label{dimensions}
Let $V$ and $W$ be $\mathbb K[x, y]$-modules that are finite dimensional over $\mathbb K$. Let the associated Lie algebras $L_V=\mathbb K\langle P,Q \rangle \rightthreetimes V$ and $L_W=\mathbb K\langle S,T \rangle \rightthreetimes W$ be isomorphic. Let $\varphi :L_V\to L_W$ be an isomorphism of Lie algebras. Write  $W_1:=\varphi (V)\cap W$ and $V_1:=\varphi ^{-1}(W_1).$ Then $V_1$ is a submodule of $V$ and $W_1$ is a submodule of $W.$ Moreover,
\begin{itemize}
  \item[\rm(i)]
if $\varphi (V)+W=L_W,$ then $\dim _{\mathbb K}V/V_1=\dim _{\mathbb K}W/W_1=2;$
  \item[\rm(ii)] if $\dim _{\mathbb K}\varphi (V)+W/W=1,$ then $\dim _{\mathbb K}V/V_1=\dim _{\mathbb K}W/W_1=1.$
\end{itemize}
\end{lemma}

\begin{proof}
We have the canonical isomorphism for $\mathbb K[x, y]$-modules
\begin{equation}\label{canon}
\varphi (V)+W/W\simeq \varphi (V)/(\varphi (V)\cap W)=\varphi (V)/W_1.
\end{equation}
 Let $\varphi (V)+W=L_W.$ By (\ref{canon}) and $\dim _{\mathbb K}L_W/W=2$, $\dim _{\mathbb K}\varphi (V)/W_1=2. $ Using the isomorphism $\varphi ^{-1}$, we get $\dim _{\mathbb K}V/V_1=2.$
Since $\dim _{\mathbb K}L_V/V_1=4$,  $\dim _{\mathbb K}L_{W}/W_1=4. $  Then $\dim _{\mathbb K}W/W_1=2.$ The statement (ii) is proved analogously.
\end{proof}

\begin{lemma}\label{dim2}
Let $V$ and $W$ be $\mathbb K[x, y]$-modules that are finite dimensional over $\mathbb K$. Let $L_V=\mathbb K\langle P,Q \rangle \rightthreetimes V$ and $L_W=\mathbb K\langle S,T \rangle \rightthreetimes W$ be the  corresponding associated Lie algebras.
Let there exist an isomorphism of Lie algebras $\varphi :L_V \to L_W$ such that $\varphi(\mathbb K\langle P,Q \rangle)+W=L_W.$ Then the $\mathbb K[x,y]$-modules $V$ and $W$ are weakly isomorphic.
\end{lemma}
\begin{proof}
Since $\varphi$ is injective and $\dim _{\mathbb K}V=\dim _{\mathbb K}W<\infty ,$ it is easy to see that $\varphi (V)\subseteq W$ implies $\varphi (V)=W.$ By Remark \ref{restriction}, we can  assume that
$\varphi(V)\nsubseteq W$.
 Since $\varphi(\mathbb K\langle P,Q\rangle)+W=L_W$, we can also assume that
 \begin{equation}\label{PQ}
 \varphi(P)=S+w_1, \qquad \varphi(Q)=T+w_2
 \end{equation}
  for some elements $w_1,w_2\in W$ (passing  to a weakly isomorphic module $W_{\theta}$, if it is needed).

We consider two cases.

\emph{Case 1:  $\varphi(V)+W=L_W$.} Write $W_1:=\varphi(V)\cap  W$ and $V_1:=\varphi^{-1}(W_1).$
By Lemma \ref{dimensions},  $V_1$ is a submodule of the $\mathbb K[x, y]$-module $V$ with  $\dim _{\mathbb K}V/V_1=2$
and $W_1$ is a submodule of $W$ with $\dim _{\mathbb K}W/W_1=2.$
 Therefore,  $V=\mathbb K\langle v_1,v_2\rangle +V_1$ for some  $v_1,v_2\in V\setminus V_1$.
Let $\varphi(v_1)=\alpha_{11}S+\alpha_{12}T+u_1$ and $\varphi(v_2)=\alpha_{21}S+\alpha_{22}T+u_2$ for some  $\alpha _{ij}\in \mathbb K$ and $ u_1,u_2\in W$. Since $\varphi (V)+W=L_W,$ the matrix $[\alpha _{ij}]$  is nonsingular. Therefore, we can choose  $v_1$ and $v_2$ such that
     \begin{equation}\label{v1,v2}
     \varphi(v_1)=S+u_1, \qquad \varphi(v_2)=T+u_2.
     \end{equation}

Since $[P,Q]=0$ in the Lie algebra $L_V$, and $\varphi$ is an isomorphism of Lie algebras, we get $$0=\varphi([P,Q])=[\varphi(P),\varphi(Q)]
=[S+w_1,T+w_2]=S(w_2)-T(w_1),$$ and so \begin{equation}\label{TS}
T(w_1)=S(w_2).
\end{equation}
Analogously, the equality $[v_1, v_2]=0$ on $L_V$ implies
\begin{equation}\label{ST}
S(u_{2})=T(u_{1}).
\end{equation}
The images of products of elements from $L_V$ in the Lie algebra $L_W$ are
\begin{equation}\label{Pv1lem2}
\varphi([P,v_1])=[S+w_1,S+u_1]=S(u_1)-S(w_1)=S(u_1-w_1).
\end{equation}
\begin{equation}\label{Qv2lem2}
\varphi([Q,v_2])=[T+w_2,T+u_2]=T(u_2)-T(w_2)=T(u_2-w_2).
\end{equation}
Using (\ref{ST}) and (\ref{TS}), we obtain
\begin{equation}\label{Pv2lem2}
\varphi([P,v_2])=[S+w_1,T+u_2]=S(u_2)-T(w_1)=S(u_2-w_2).
\end{equation}
\begin{equation}\label{Qv1lem2}
\varphi([Q,v_1])=[T+w_2,S+u_1]=T(u_1)-S(w_2)=T(u_1-w_1).
\end{equation}

The equality $[V, V_1]=0$ in the Lie algebra $L_V$ implies
$$\varphi([V, V_1])=0=[\varphi(V),\varphi(V_1)]=[\varphi(V), W_1] .$$
 Since $[W, W_1]=0$  and $L_W=\varphi (V)+W,$ we obtain
$[L_W, W_1]=0,$ which implies $[L_V, V_1]=0.$  The last two equalities ensure
\begin{equation}
[P, V_1]=[Q, V_1]=0, \qquad [S, W_1]=[T, W_1]=0.
\end{equation}

Let us show that the linear operators $P$ and $Q$ act trivially on the quotient module $V/V_1.$ Let us consider, for example,  $P(v_1)=[P, v_1].$
 Since $P(v_1)\in V$ and $\varphi(P(v_1))\in W$ (by \eqref{Pv1lem2}), we have $\varphi(P(v_1))\in W_1$. Using $V_1=\varphi^{-1}(W_1)$, we find that $P(v_1)\in V_1$.
Using (\ref{Pv2lem2}),
 (\ref{Qv1lem2}), and (\ref{Qv2lem2}), we analogously prove that $P(v_2), Q(v_1), Q(v_2)\in V_1$.

 Define  the linear map $\widehat{\varphi}:V \to W$ as follows:
$\widehat{\varphi}(v)=\varphi(v)$ for $v\in V_1$, $\widehat{\varphi}(v_1)=u_1-w_1$, and $\widehat{\varphi}(v_2)=u_2-w_2$. It is easy to check  for all $v\in V_1$ that
$$\widehat{\varphi}(x\cdot v)=\widehat{\varphi}(P(v))=0=S(\varphi(v))=x\cdot\widehat{\varphi}(v), $$
 $$ \widehat{\varphi}(x\cdot v_1)=\widehat{\varphi}(P(v_1))=\varphi(P(v_1))=S(u_1-w_1)=x\cdot(u_1-w_1)=x\cdot \widehat{\varphi}(v_1).$$ Analogous relations hold for $x\cdot v_2$,  $y\cdot v_1$, and $y\cdot v_2$. Therefore, the linear mapping $\widehat{\varphi}$ is a homomorphism of $\mathbb K[x, y]$-modules $V$ and $W.$

Let us show that $\widehat{\varphi}$ is a surjective homomorphism. It is sufficient to verify that $u_1-w_1$ and $u_2-w_2$ are linearly independent over $W_1$ in the vector space $W$.  Conversely, suppose that $\alpha(u_1-w_1)+\beta(u_2-w_2)\in W_1$ for some $\alpha, \beta \in \mathbb K$ such that at least one of them is nonzero. Then
$$\varphi(\alpha (v_1-P)+\beta (v_2-Q))=\alpha (u_1-w_1)+\beta (u_2-w_2)\in W_1,$$
$\alpha (v_1-P)+\beta (v_2-Q)\in V_1$, and therefore
 $\alpha P+\beta Q\in V,$ which is impossible because  $\mathbb K\langle P, Q\rangle \cap  V=0.$ We have proved that $\widehat{\varphi}$ is surjective. Since the modules $V$ and $W$ have the same dimension over $\mathbb K,$ the mapping $\widehat{\varphi}$ is an isomorphism between  $\mathbb K[x, y]$-modules $V$ and  $W.$

\emph{Case 2: $\varphi(V)+W\neq L_W$.}  Then $\dim _{\mathbb K}(\varphi(V)+W)/W=1.$  By Lemma \ref{dimensions},  $W_1=\varphi(V)\cap W$ is a submodule of  $W$ with $\dim _{\mathbb K}W/W_1=1$, and $V_1=\varphi^{-1}(W_1)$ is a submodule of $V$ with  $\dim _{\mathbb K}V/V_1=1.$ Take any $v_0\in V\setminus V_1.$ Then $V=\mathbb K\langle v_{0}\rangle +V_1$ and $\varphi (v_0)=\alpha S+\beta T+w_0$ for some  $\alpha, \beta \in \mathbb K$ and $ w_0 \in W$. As earlier,  we assume that  $\varphi (P)=S+w_1, \varphi (Q)=T+w_2$.
 Then
 $$\varphi([P,v_o])=[\varphi(P),\varphi(v_0)]=[S+w_1, \alpha S+\beta T+w_0]=S(w_0)-\alpha S(w_1)-\beta T(w_1),$$
  $$\varphi([Q,v_o])=[\varphi(Q),\varphi(v_0)]=[T+w_2, \alpha S+\beta T+w_0]=T(w_0)-\alpha S(w_2)-\beta T(w_2).$$
Using (\ref{TS}),  we get
 \begin{equation}\label{Pv0}
 \varphi([P,v_o])=S(w_0-\alpha w_1-\beta w_2),  \ \ \ \varphi([Q,v_o])=S(w_0-\alpha w_1-\beta w_2). \end{equation}
These equalities imply $\varphi(P(v_0)),\varphi(Q(v_0))\in W$. As in Case 1, $P(v_0),Q(v_0)\in V_1$.
Define the linear map $\widetilde \varphi:V \to W$ by  $\widetilde \varphi(v)=\varphi(v)$ for $v\in V_1$ and $\widetilde \varphi(v_0)=w_0-\alpha w_1-\beta w_2$. By (\ref{Pv0}), $\widetilde\varphi$ is a homomorphism of $\mathbb K[x,y]$-modules (note that $[P, v_1], [Q, v_1]\in V_1$ for all $v_1\in V_1$).

Let us show that $w_0-\alpha w_1-\beta w_2\not \in W_1$. Conversely, let there exist $v\in V_1$ such that $w_0-\alpha w_1-\beta w_2=\varphi(v)$. Then
$$\varphi(v_0-v)=\alpha S+\beta T+w_0-w_0+\alpha w_1+\beta w_2=\alpha(S+w_1)+\beta(T+w_2)=\varphi(\alpha P+\beta Q).$$
Since  $\varphi$ is an isomorphism of Lie algebras, these equalities imply  $v_0-v=\alpha P+\beta Q.$  Then $v_0-v\in V\cap \mathbb K\langle P,Q \rangle =0$. Hence  $v=v_0$ and $\varphi(v_0)\in W_1$, which contradicts our choice of $v_0$. Thus, $w_0-\alpha w_1-\beta w_2\not \in W_1$, $\widetilde \varphi$ is a surjective homomorphism, and so it is an isomorphism of $\mathbb K[x,y]$-modules.
\end{proof}

\begin{lemma}\label{example}
Let $V=\mathbb K\langle v_1, v_2, a_1, a_2 \rangle$ and $W=\mathbb K\langle w_1, w_2, b_1, b_2\rangle$ be modules over the ring $\mathbb K[x,y]$ with the following actions  of $x$ and $y$:

\item[\rm(i)] $xv_1=a_1$, $xv_2=a_2,$ and the  other products of $x, y$ and the basis elements of $V$ are zero,
\item[\rm(ii)] $xw_1=b_1$, $yw_1=b_2,$ and the   other products of $x, y$ and the basis elements of $W$ are zero.

Then the $\mathbb K[x,y]$-modules $V$ and $W$ are not weakly isomorphic, but the corresponding associated Lie algebras $L_V$ and $L_W$ are isomorphic.
\end{lemma}

 \begin{proof}
The annihilator $\Ann_W(\mathbb K[x,y]):=\{w\in W  \ | \ \mathbb K[x,y]\cdot w=0\}$ of $\mathbb K[x, y]$ in $W$ is a submodule of  $W.$ Then  $\Ann_W(\mathbb K[x,y])=\mathbb K\langle w_2,b_1,b_2\rangle$ and  $$\Ann_{W_{\theta}}(\mathbb K[x,y])=\Ann_{W}(\mathbb K[x,y])=\mathbb K\langle w_2, b_1,b_2\rangle$$
for any linear automorphism $\theta \in \Aut(\mathbb K[x,y])$.

Let us show that $\Ann_{V}(\mathbb K[x,y])=\mathbb K\langle a_1,a_2\rangle$. Take any $v\in V$ and write $v=\gamma_1v_1+\gamma_2v_2+z$, where $\gamma_1,\gamma_2 \in \mathbb K$ and $z\in \mathbb K\langle a_1,a_2\rangle$. If $v\in \Ann_{V}(\mathbb K[x,y])$, then $x\cdot v=0$ and therefore
   $\gamma_1a_1+\gamma_2a_2=0$ because $x\cdot z=0.$  Since $a_1$ and $a_2$ are linearly independent over $\mathbb K$, we get $\gamma_1=\gamma_2=0.$ Thus $v=z\in \mathbb K\langle a_1,a_2\rangle .$
    Clearly, $\mathbb K\langle a_1,a_2\rangle\subseteq \Ann_{V}(\mathbb K[x,y])$. Therefore, $\mathbb K\langle a_1,a_2\rangle= \Ann_{V}(\mathbb K[x,y])$.

Now suppose that the $\mathbb K[x,y]$-modules $V$ and $W$ are weakly isomorphic. Then there exists $\mathbb K[x,y]$-module $U$ such that $V\simeq U$ and $W=U_{\theta}$ for some linear automorphism  $\theta$ of the ring $\mathbb K[x, y].$ Since $V\simeq U$, $\dim_{\mathbb K}\Ann_U(\mathbb K[x,y])=2$. It is easy to see that $\Ann_{U}(\mathbb K[x,y])=\Ann_{U_{\theta}}(\mathbb K[x,y]).$
        Then $\Ann_W(\mathbb K[x,y])$ is of dimension 2 over $\mathbb K,$ which is impossible. The obtained contradiction shows that $V$ and $W$ are not weakly isomorphic.

Let $L_V$ and $L_W$ be the Lie algebras associated with $\mathbb K[x,y]$-modules $V$ and $W$. Define the linear map $\overline \varphi:L_V\to L_W$ by $$\overline \varphi(P)=-w_1,\mbox{ } \overline \varphi(Q)=-w_2, \ \overline \varphi(v_1)=S,\mbox{ } \overline \varphi(v_2)=T,  \  \overline \varphi(a_1)=b_1, \mbox{ } \overline \varphi(a_2)=b_2.$$ Then
$$\overline \varphi([P,v_1])=\overline \varphi(P(v_1))=\overline \varphi(a_1)=b_1=S(w_1)=[S,w_1]=[-w_1,S]=[\overline \varphi(P),\overline \varphi(v_1)],$$
$$\overline \varphi([Q,v_1])=\overline \varphi(Q(v_1))=\overline \varphi(0)=0=S(w_2)=[-w_2,S]=[\overline \varphi(Q),\overline \varphi(v_1)].$$

We analogously check that $\overline \varphi([P,v_2])=[\overline \varphi(P),\overline \varphi(v_2)]$ and $\overline \varphi([Q,v_2])=[\overline \varphi(Q),\overline \varphi(v_2)]$. Since the other products of basic elements are zero, we find that $\overline \varphi$ is an isomorphism of the Lie algebras $L_V$ and $L_W$.
 \end{proof}

\begin{lemma}\label{dim6}
Let $V$ and $W$ be $\mathbb K[x, y]$-modules that are finite dimensional over $\mathbb K$ and that are not weakly isomorphic.
Let $L_V=\mathbb K\langle P,Q \rangle \rightthreetimes V$ and  $L_W=\mathbb K\langle S,T \rangle \rightthreetimes W$ be their  associated Lie algebras. Let there exist an isomorphism $\varphi: L_V \to L_W$ of Lie algebras such that $\varphi(\mathbb K\langle P,Q\rangle)\subseteq W$. Then $V=V_0\oplus V_2$ and $W=W_0\oplus W_2$, where $V_2$ and $W_2$ are isomorphic $\mathbb K[x,y]$-modules with the trivial action of  $\mathbb K[x,y]$ on them, and where $V_0$ and $W_0$ are not weakly isomorphic $\mathbb K[x,y]$-modules with $\dim _{\mathbb K}V_0=\dim _{\mathbb K}W_0 \leq 6$.
\end{lemma}
 \begin{proof}
We have $\varphi(V)+W=L_W$ since otherwise $$\varphi(L_V)\subseteq \varphi(\mathbb K\langle P,Q\rangle)+\varphi(V)\subseteq \varphi(V)+W\subset L_W,$$ which is impossible by $\varphi(L_V)=L_W$.
Write $W_1:=\varphi(V)\cap W$  and  $V_1:=\varphi^{-1}(W_1).$ By Lemma \ref{dimensions}, $W_1$ is a submodule of codimension $2$ in $W$, and $V_1$ is a submodule of codimension $2$ in $V.$
Take any $v_1, v_2\in V\setminus V_1$ such that  $V=\mathbb K\langle v_1, v_2\rangle +V_1.$
By conditions of the lemma, $\varphi(\mathbb K\langle P,Q\rangle)\subseteq W,$ and so  $\varphi(P)=w_1$ and $\varphi(Q)=w_2$ for some  $w_1, w_2\in W.$  Since $\varphi (V)+W=L_W,$ there are $v_1, v_2\in V$ such that $\varphi(v_1)=S+u_1$ and $\varphi(v_2)=T+u_2$, where $u_1, u_2\in W.$ Since $[v_1, v_2]=0$, we get
$$\varphi([v_1,v_2])=[\varphi(v_1),\varphi(v_2)]=[S+u_1,T+u_2]=S(u_2)-T(u_1)=0,$$ which implies $S(u_2)=T(u_1)$.  Similarly, \begin{equation}\label{Pv1lem4}
\varphi([P,v_1])=[\varphi(P),\varphi(v_1)]=[w_1, S+u_1]=-S(w_1),\end{equation}
\begin{equation}\label{Pv2lem4}
\varphi([P,v_2])=[\varphi(P),\varphi(v_2)]=[w_1, T+u_2]=-T(w_1),\end{equation}
\begin{equation}\label{Qv1lem4}
\varphi([Q,v_1])=[\varphi(Q),\varphi(v_1)]=[w_2, S+u_1]=-S(w_2),\end{equation}
\begin{equation}\label{Qv2lem4}
\varphi([Q,v_2])=[\varphi(Q),\varphi(v_2)]=[w_2, T+u_2]=-T(w_2).\end{equation}
It follows from $[v_1,V_1]=[v_2,V_1]=0$ that
\begin{align*}
\varphi([v_1,V_1])=[S+u_1,W_1]=[S,W_1]=0, \\ \varphi([v_2,V_1])=[T+u_2,W_1]=[T,W_1]=0.
\end{align*}
Hence $W_1\subseteq Z(L_W)$ and therefore  $V_1=\varphi^{-1}(W_1)\subseteq Z(L_V)$, where $Z(L_V)$ and $Z(L_W)$ are the centers of the Lie algebras $L_V$ and $L_W$.

Let us show that the derived subalgebra $L_V'$ of the Lie algebra $L_V$ is of dimension $\leq 4$ over $\mathbb K.$
All $g_1, g_2\in L$ are represented in the form $$g_1=\alpha_{1}P+\beta_{1}Q+\gamma_{1}v_1+\delta_{1}v_2+u_3,\qquad g_2=\alpha_{2}P+\beta_{2}Q+\gamma_{2}v_1+\delta_{2}v_2+u_4,$$ where  $\alpha_{i},\beta_{i},\gamma_i,\delta_i\in \mathbb K$, $i=1,2$, and $u_3,u_4\in V_1$. We have $$[g_1, g_2]=\alpha_{11}[P,v_1]+\alpha_{12}[P,v_2]+\alpha_{21}[Q,v_1]+\alpha_{22}[Q,v_2]$$ for some $\alpha_{ij}\in \mathbb K$.  Thus, $L'_V$ is a $\mathbb K$-linear hull of the elements $[P,v_1],\ [P,v_2],\ [Q,v_1]$, and $ [Q,v_2]$, and so $\dim _{\mathbb K}L'_V\leq 4$.

 Consider the  subalgebra $V_0$ of the Lie algebra $L_V$ generated by $v_1, \  v_2$, and their images under the action of the operators $P$ and $Q.$   It is easy to show that  $V_0$ is a submodule of $\mathbb K[x,y]$-module $V$, and $V_0$ is the linear hull of elements   $v_1,\ v_2,\ P(v_1),\ P(v_2),\ Q(v_1)$, and $ Q(v_2),$  and so  $\dim _{\mathbb K}V_{0}\leq 6.$
 Take any subspace $V_2$ of the vector space $V_1$ such that $V_1=(V_0\cap V_1)\oplus V_2$.
  It is easy to see that $V=V_0\oplus V_2$ and that $\mathbb K[x,y]$ acts trivially on $V_2$  since $V_2\subset V_1$ and $[P,V_1]=[Q,V_1]=0$. Thus, $V=V_0\oplus V_2$ is a direct sum of $\mathbb K[x, y]$-submodules.

Analogously, we  consider the subalgebra $W_0$ of the Lie algebra $L_W$ that is generated by $w_1,\ w_2$, and their images  under the action of $S$ and $T$. It is clear that $L'_W$  coincides with the $\mathbb K$-linear hull of elements $[S,w_1]$, $[S,w_2]$, $[T,w_1]$, and $[T,w_2]$. Since $L_V$ and $L_W$ are isomorphic Lie algebras, $\dim _{\mathbb K}L'_V=\dim _{\mathbb K}L'_W\leq 4$. Furthermore, $W_0$ is a submodule of $W$ and $\dim _{\mathbb K}V_0=\dim _{\mathbb K}W_0\leq 6$. Take any direct complement  $W_2$ of $W_0\cap W_1$ in the $\mathbb K[x, y]$-module $W_1$. Then $W=W_0\oplus W_2$ is a direct sum of submodules, and $\mathbb K[x,y]$ acts on $W_2$ trivially since $W_2\subset W_1$. Using $\dim _{\mathbb K}V_2=\dim _{\mathbb K}W_2,$ we get  that $V_2$ and $W_2$ are isomorphic $\mathbb K[x,y]$-modules. By the conditions of lemma, the $\mathbb K[x,y]$-modules $V$ and $W$ are not weakly isomorphic, and so $V_0$ and $W_0$ are not weakly isomorphic too.
\end{proof}

 \begin{lemma}\label{dim1}
Let $V$ and $W$ be $\mathbb K[x, y]$-modules that are finite dimensional over $\mathbb K$.
Let  $L_V=\mathbb K\langle P,Q \rangle \rightthreetimes V$ and $L_W=\mathbb K\langle S,T \rangle \rightthreetimes W$ be  the corresponding associated Lie algebras. Let there exist an isomorphism $\varphi: L_V \to L_W$ of Lie algebras such that $\dim_{\mathbb K}(\varphi(V)+W)/W=1$ and $\dim_{\mathbb K}(\varphi(\mathbb K\langle P,Q \rangle)+W)/ W=1.$ Then the $\mathbb K[x, y]$-modules $V$ and $W$  are weakly isomorphic.
 \end{lemma}

\begin{proof}
Write  $W_1:=\varphi(V)\cap W$  and $V_1:=\varphi^{-1}(W_1).$  By Lemma \ref{dimensions},   $W_1$ is a submodule  of codimension 1 in $W$   and  $V_1$ is a submodule of codimension 1 in $V$.
Take any $v_1\in V\setminus V_1.$ Then $\mathbb K\langle v_1\rangle +V_1=V.$  Passing to a weakly isomorphic module $W_{\theta}$, if it is needed, we can assume that $\varphi(v_1)=S+u_1$ for some  $u_1\in W.$ The condition $\dim _{\mathbb K}(\mathbb K\langle P, Q\rangle  +W)/W=1$ implies that $\varphi (\mathbb K\langle P, Q\rangle )\cap W\ne 0,$ and therefore $\varphi (\alpha P+\beta Q)=w_1\in W$ for some $\alpha , \beta \in \mathbb K$ and nonzero $w_1\in W.$ Passing  to a weakly isomorphic module $V_{\sigma}$, if it is needed, we can assume that $\varphi (P)=w_1$ and get
\begin{equation}\label{v1P}
\varphi (v_1)=S+u_1, \qquad \varphi (P)=w_1.
\end{equation}
Since $V$ is an abelian subalgebra of $L_V$, $[v_1,V_1]=0$. Using $\varphi(V_1)=W_1$, we obtain $$\varphi([v_1,V_1])=[\varphi(v_1),\varphi(V_1)]=[S+u_1, W_1]=[S, W_1]=0.$$
  Similarly, $[w_1,W_1]=0$ implies $\varphi^{-1}([w_1, W_1])=[P,V_1]=0$. Since $P\notin V_1$, $\varphi (P)=w_1\notin  W_1. $ Therefore, $W=\mathbb K\langle w_1\rangle +W_1.$

Without loss of generality, we can assume that $u_1\in \mathbb K\langle w_1 \rangle ,$ where $u_1$ is defined in   (\ref{v1P}).  Indeed, $u_1=\alpha_1 w_1+w_2$ for some $\alpha_1\in \mathbb K$ and $w_2\in W_1$. If $w_2\neq 0$, then we can take $v_1-v_2$ instead of $v_1$ for some $v_2\in V_1$ such that $\varphi(v_2)=w_2$ (it is possible because $\varphi: V_1 \to W_1$ is a bijection). Therefore, $u_1=\alpha_1 w_1.$
Further, $\varphi (Q)=\gamma S+\delta T+u_2$ for some $\gamma , \delta \in \mathbb K,\ u_2\in W.$ Note that $\delta \ne 0$ since otherwise $\varphi (\mathbb K\langle P, Q, v_1, V_{1}\rangle )\subseteq \mathbb K\langle S\rangle +W\not =L_W$,  which is impossible. We can assume that $\varphi (Q)=T+u_3$ for some $u_3\in W$  (passing  to a module $W_{\pi}$ if it is needed, where $\pi (x)=x, \pi (y)=y/\delta -\gamma x/\delta$).  Moreover, replacing $Q$ by $Q'=Q-\mu P$ for some $\mu \in \mathbb K$ and using $\varphi (P)=w_1$, we can assume that $u_3\in W_1$. This means that we go to the module $V_{\rho}$ with the automorphism $\rho$ defined by $\rho (x)=x$ and $ \rho (y)=y-\mu x.$  The last two automorphisms save the relations (\ref{v1P}), so we can use them in the sequel. By
$$ 0=\varphi([P,Q])=[\varphi (P),\varphi (Q)]=[w_1,T+u_3]=-T(w_1),$$ we have
\begin{equation}\label{Tw1}
T(w_1)=0.
\end{equation}
We get
\begin{equation}\label{Pv1lem5}
\varphi([P,v_1])=[\varphi(P),\varphi(v_1)]=[w_1, S+ u_1]=-S(w_1),
\end{equation}
\begin{equation}\label{Qv1lem5}
\varphi([Q,v_1])=[\varphi(Q),\varphi(v_1)]=[T+u_3, S+u_1]=T(u_1)-S(u_3)=0.
\end{equation}
The last equation in (\ref{Qv1lem5}) holds because $T(u_1)=0$ (by (17)) and $S(u_3)=0$, since $u_3\in W_1$ and $[S, W_1]=0.$
By $\varphi ([Q, v_1])=0$,  $[Q, v_1]=0.$
Let us show that $P(v_1)\in V_{1}.$
Indeed, by (\ref{Pv1lem5}) $\varphi(P(v_1))=-S(w_1)\in W$. Moreover, $P(v_1)\in V$ and thus $\varphi(P(v_1))\in \varphi (V)$. Then $\varphi (P(v_1))\in W_1$ and $P(v_1)\in \varphi^{-1}(W_1)=V_1$.
Define the linear map $\psi: V\to W$ by  $\psi(v)=\varphi(v)$ for $v\in V_1$ and $\psi(v_1)=-w_1$. Then the restriction of $\psi$ on $V_1$ is an isomorphism of $\mathbb K[x, y]$ modules $V_1$ and $W_1.$ Using (\ref{Pv1lem5}) and (\ref{Qv1lem5}), we find that $\psi$ is an isomorphism of $\mathbb K[x,y]$-modules $V$ and $W.$
\end{proof}

\section{The main theorems}

\begin{theorem}\label{structure}
Let $V$ and $W$ be $\mathbb K[x, y]$-modules that are finite dimensional over $\mathbb K$. Let  $L_V=\mathbb K\langle P,Q \rangle \rightthreetimes V$ and $ L_W=\mathbb K\langle S,T \rangle \rightthreetimes W$ be  their associated Lie algebras. Let $L_V$ and $L_W$  be isomorphic. Then one of the following conditions holds:
\begin{itemize}
\item[(i)] $V$ and $W$ are weakly isomorphic $\mathbb K[x,y]$-modules;
\item[(ii)] $V=V_0\oplus V_2$  and $W=W_0\oplus W_2$, where $V_0,\mbox{ }W_0$ are not weakly isomorphic submodules  such that $\dim _{\mathbb K}V_0=\dim _{\mathbb K}W_0 \leq 6$, and where $V_2$ and $W_2$ are submodules of equal dimension such that $\mathbb K[x,y]$ acts trivially on them.
\end{itemize}
\end{theorem}
\begin{proof}
Let $\varphi: L_V \to L_W$ be an isomorphism of Lie algebras. If $\varphi(V)=W$ or $\varphi(\mathbb K\langle P,Q\rangle)+W=L_W$, then $V$ and $W$ are weakly isomorphic modules by Remark~\ref{remark} and Lemma~\ref{dim2}. If $\dim_{\mathbb K}\varphi(V)+W/W=1$ and $\dim_{\mathbb K}\varphi(\mathbb K \langle P,Q\rangle) +W/W=1$, then the modules $V$ and $W$ are weakly isomorphic  by Lemma~\ref{dim1}. If $\varphi(\mathbb K \langle P,Q\rangle)\subseteq W$ and the modules $V$ and $W$ are not weakly isomorphic, then $V$ and $W$ are of type (ii) of the theorem by Lemma (\ref{dim6}).

Thus, we can assume that $\dim_{\mathbb K}(\varphi(\mathbb K \langle P,Q\rangle) +W)/W=1$  and $\varphi(V)+W=L_W$. By Lemma \ref{dimensions}, $W_1=\varphi(V)\cap W$ is a submodule of codimension $2$  in $W$ and $V_1=\varphi^{-1}(W_1)$ is a submodule of codimension $2$ in $V.$  Choose $v_1,v_2\in V\setminus V_1$ such that $V=\mathbb K\langle v_1,v_2 \rangle +V_1$ and
\begin{equation}\label{v1,v2th}
\varphi(v_1)=S+u_1, \qquad  \varphi(v_2)=T+u_2
\end{equation}
for some $u_1, u_2\in W$ (which is possible since $v_1$ and $v_2$ are linearly independent over $V_1$, and $\varphi(V)+W=L_W$). As in the proof of Lemma \ref{dim6}, we find that $[P,V_1]=[Q,V_1]=0$ and $[S,W_1]=[T,W_1]=0$.

Since $\dim _{\mathbb K}\varphi(\mathbb K \langle P,Q\rangle) +W/W=1$, we get $\varphi(\mathbb K\langle P,Q\rangle)\cap W\neq 0$. Take a nonzero  $w_3\in \varphi(\mathbb K\langle P,Q\rangle)\cap W$ and write $\varphi(\alpha P+\beta Q)=w_3$ for some $\alpha, \beta\in \mathbb K$. We can assume that $\varphi(Q)=w_3$ (passing to a weakly isomorphic module $V_{\theta}$ if it is needed). Moreover, $\varphi(\mathbb K\langle P,Q\rangle)\not\subseteq W$, so we assume that $\varphi(P)=S+u_3$ for some $u_3\in W$ (passing to a weakly isomorphic module $W_{\sigma}$ if it is needed).

Since $[P,Q]=0$ and $[v_1,v_2]=0$, we get $$[\varphi(P),\varphi(Q)]=[S+u_3,w_3]=S(w_3)=0$$ and
$$[\varphi(v_1),\varphi(v_2)]=[S+u_1,T+u_2]=S(u_2)-T(u_1)=0.$$ As a consequence,
\begin{equation}\label{SST}
S(w_3)=0, \qquad S(u_2)=T(u_1).
\end{equation}
Similarly,
\begin{equation}\label{Pv1th}
\varphi([P,v_1])=[\varphi(P),\varphi(v_1)]=[S+u_3, S+u_1]=S(u_1-u_3),
\end{equation}
\begin{equation}\label{Pv2th}
\varphi([P,v_2])=[S+u_3, T+u_2]=S(u_2)-T(u_3)=-T(u_1-u_3),
\end{equation}
\begin{equation}\label{Qv1th}
\varphi([Q,v_1])=[\varphi(Q),\varphi(v_1)]=[w_3, S+u_1]=-S(w_3)=0,
\end{equation}
\begin{equation}\label{Qv2th}
\varphi([Q,v_2])=[\varphi(Q),\varphi(v_2)]=[w_3, T+u_2]=-T(w_3)
\end{equation}
(we use (\ref{SST}) in (\ref{Pv2th}) and (\ref{Qv1th})).

Since  $\varphi$ is an isomorphism of the Lie algebras $L_V$ and $L_W$, and $\varphi ([Q, v_1])=0$ by (\ref{Qv1th}), we get $[Q, v_1]=0.$
Observe that $\mathbb K\langle u_1-u_3, w_3\rangle +W_1=W$. Indeed, it is easy to see that $u_1-u_3=\varphi(v_1-P)$. Since $v_1-P$ and $Q$ are linear independent in $L_V$ over $V_1$, their images $u_1-u_3$ and $w_3$ under $\varphi$ are linearly independent over $W_1$ in $W$.

Write $V_0:=\mathbb K\langle v_1,v_2,P(v_1),P(v_2),Q(v_2)\rangle$. It is easy to show that  $V_0$ is a submodule of the $\mathbb K[x,y]$-module $V$ and $\dim _{\mathbb K}V_0\leq 5$. Take any $\mathbb K$-submodule $V_2\subset V_1$ such  that  $V_1=(V_0\cap V_1)\oplus V_2$ (this submodule exists because $P$ and $Q$ act trivially on $V_1$).  Then $V=V_0\oplus V_2$ is a direct sum of $\mathbb K[x, y]$-submodules. Similarly, let us consider the  submodule $W_0:=\mathbb K\langle u_1-u_3,w_3, S(u_1-u_3), S(w_3), T(u_1-u_3)\rangle$ of $W$ and take a submodule $W_2$ of $W_1$ such that $W_{1}=(W_0\cap W_1)\oplus W_2$ (it exists since $T$ and  $S$ act trivially on $W_1$).
  Then $W=W_0\oplus W_2$ is a direct sum of submodules, and $\mathbb K[x,y]$ acts trivially on  $ W_2.$ Furthermore, $\dim _{\mathbb K} V_0=\dim _{\mathbb K}W_0\leq 5$ and $\dim _{\mathbb K}V_2=\dim _{\mathbb K}W_2$. Thus, $V_2$ and $W_2$ are isomorphic $\mathbb K[x,y]$-modules. If $V$ and $W$ are not weakly isomorphic, then the $\mathbb K[x,y]$-modules $V_0$ and $W_0$ are also not weakly isomorphic. We see that $V$ and $W$ are of type (ii) of the theorem.
\end{proof}
The following theorem is a direct consequence of the previous results.
\begin{theorem}\label{indecomposable}
Let $V$ and $ W$ be indecomposable modules over the ring $\mathbb K[x, y]$. Let $L_{V}=\mathbb K\langle P, Q\rangle \rightthreetimes V $ and $  L_{W}=\mathbb K\langle S, T\rangle \rightthreetimes W$ be their associated Lie algebras. Let $\dim _{\mathbb K}V=\dim _{\mathbb K}W\geq 7.$ Then the $\mathbb K[x, y]$-modules $V$ and $W$ are weakly isomorphic if and only if the Lie algebras $L_V$ and $ L_W$ are isomorphic.

\end{theorem}

\begin{corollary}
The problem of classifying finite dimensional  Lie algebras of the form $L=B\rightthreetimes A$  with an  abelian ideal $A$ and a two-dimensional abelian subalgebra $B$ is equivalent to the problem of classifying finite dimensional $\mathbb K[x, y]$-modules  up to weak isomorphism.
\end{corollary}

The authors are grateful to V.V.Sergeichuk for useful discussions and advice.


\begin{thebibliography}{99}

\bibitem{BBLPS}
G. Belitskii, V.M. Bondarenko, R. Lipyanski, V.V. Plachotnik, V.V. Sergeichuk, The problems of
classifying pairs of forms and local algebras with zero cube radical are wild, Linear Algebra Appl.
402 (2005) 135–-142.

\bibitem{BLS} G.~Belitskii, R.~Lipyanski, V.~V.~Sergeichuk, Problems of classifying associative or Lie algebras and triples of symmetric or skew-symmetric matrices are wild, Linear Algebra Appl. 407 (2005) 249--262.

 \bibitem{GP}
I.~M.~Gelfand, V.~A.~Ponomarev,  Remarks on the classification of a pair of commuting linear transformations in a finite dimensional space, Functional Anal. Appl. 3 (1969) 325–-326.

 \bibitem{Qui}  D. Quillen, Projective modules over polynomial rings, Invent. Math. 36 (1976)
167--171.

\bibitem{Sus} A.~Suslin, Projective modules over polynomial rings are free, Soviet Math. Doklady 17 (1976) 1160–-1164.


\end{thebibliography}
\end{document}